\documentclass{amsart}

\usepackage{amssymb}
\usepackage[all]{xy}
\usepackage{palatino}

\def\eu{\mathfrak}
\def\ma{\mathbb}

\def\Q{\mathbb Q}
\def\Z{\mathbb Z}
\def\N{\mathbb N}
\def\J{{\mathcal J}}
\def\I{{\mathcal I}}

\renewcommand{\gcd}{{\rm gcd}}
\renewcommand{\deg}{{\rm \,deg}\,}

\def\S#1{S_{\infty}(#1)}
\def\e#1{e_{\infty}(#1)}
\def\f#1{f_{\infty}(#1)}
\def\p{{\mathcal P}_{\infty}}
\def\g#1{#1_{{\eu {ge}}}}

\def\F{{\ma F}_q}

\def\lam#1{k(\Lambda_#1)}

\def\smid{\,|\,}

\newcommand{\Gal}{\operatorname{Gal}}

\newcounter{bean}

\def\l{
\begin{list}
{\rm{(\alph{bean}).-}}{\usecounter{bean}
\setlength{\labelwidth}{0.8in}
\setlength{\labelsep}{0.3cm}
\setlength{\leftmargin}{1cm}}}

\theoremstyle{plain}
\newtheorem{teo}{Theorem}[section]
\newtheorem{prop}[teo]{Proposition}
\newtheorem{cor}[teo]{Corollary}

\newtheorem{ejm}[teo]{Example}

\newtheorem{no}[teo]{Remark}

\title[Genus fields of Kummer $\ell^n$--cyclic extensions]
{Genus fields of Kummer $\ell^n$--cyclic extensions}

\author[C. Reyes]{Carlos Daniel Reyes--Morales}
\address{Departamento de Control Autom\'atico\\
Centro de Investigaci\'on y de Estudios Avanzados del I.P.N.}
\email{mcenigm@gmail.com, creyes@ctrl.cinvestav.mx}

\author[G. Villa]
{Gabriel Villa--Salvador}
\address{
Departamento de Control Autom\'atico\\
Centro de Investigaci\'on y de Estudios Avanzados del I.P.N.}
\email{gvillasalvador@gmail.com, gvilla@ctrl.cinvestav.mx}



\date{June 21, 2020}

\begin{document}

\begin{abstract}

We give a construction of the genus field for Kummer $\ell^n$-cyclic extensions of rational congruence function fields, where $\ell$ is a prime number. First, we compute the genus field of a field contained in a cyclotomic function field, and then for the general case.
This generalizes the result obtained by Peng for a Kummer $\ell$-cyclic extension. Finally, we study the extension $(K_1K_2)_{\eu{ge}}/(K_1)_{\eu{ge}}(K_2)_{\eu{ge}}$, for $K_1$, $K_2$ abelian extensions of $k$.

\end{abstract}

\maketitle

\section{Introduction}\label{S1}

The origin of genus fields dates back to C. F. Gauss \cite{Gau1801} in his work about binary quadratic forms. For a finite field extension $K$ of $\Q$, the rational number field, the genus field $\g K$ is defined as the maximum unramified field extension of $K$ such that it is the composite $Kk^{\ast}$, where $k^{\ast}$ is an abelian field extension over $\Q$. This definition is due to A. Fr\"ohlich \cite{Fro83}. We have $K\subseteq K_{{\eu {ge}}}\subseteq K_H$, where $K_H$ denotes the Hilbert class field over $K$. Originally the concept of genus fields was given for quadratic field extensions of ${\Q}$. 

Using Dirichlet characters, H. W. Leopoldt \cite{Leo53} determined the narrow genus field $K_{{\eu {gex}}}$ of a finite abelian field extension $K$ over ${\Q}$, i.e, $K_{{\eu {gex}}}$ is the maximum abelian field extension of $\Q$ such that $K_{{\eu {gex}}}/K$ is unramified at any finite prime of $K$. This generalizes the results of H. Hasse \cite{Has51} who introduced genus theory for quadratic extensions of number fields.

M. Ishida described the narrow genus field $K_{{\eu {gex}}}$ of any finite extension of ${\Q}$ \cite{Ish76}. X. Zhang \cite{Xia85} gave a simple expression of $K_{{\eu {ge}}}$ of any finite abelian extension $K$ of $\Q$ using Hilbert ramification theory.

For function fields, the notion of Hilbert class field as the maximum unramified abelian extension of a congruence function field $K/{\ma F}_q$ is not suitable since it contains all the constant field extensions $K_m:=K{\ma F}_{q^m}$ for every natural number $m$. Therefore the maximum unramified abelian extension of $K$ is of infinite degree over $K$.

M. Rosen \cite{Ros87} gave a definition of an analogue of the Hilbert class field of $K$ for a fixed finite nonempty set $S$ of prime divisors of $K$. Given a finite nonempty set $S$ of places of a global function field $K$, the Hilbert class field (relative to $S$) $K_{H,S}$ of $K$ is defined as the maximum unramified abelian extension of  $K$ such that the places in $S$ are completely descomposed in $K_{H,S}$. Using Rosen's definition of Hilbert class field, it is possible to give a proper concept of genus fields along the lines of number fields.

R. Clement \cite{Cle92} found a narrow genus field of a cyclic extension of $k:={\ma F}_q(T)$ of prime degree $\ell$ dividing $q-1$. She found the genus field using class field theory and defining the Hilbert class field following the ideas of H. Hasse \cite{Has51}. Later, S. Bae and J. K. Koo \cite{BaeKoo96} were able to generalize the results of Clement with the methods developed by Fr\"ohlich \cite{Fro83}.

G. Peng \cite{Pen2003} explicitly described the genus theory for Kummer function fields extensions of prime degree. Later S. Hu and Y. Li \cite{HuLi2010} described explicitly the ambiguous ideal  classes and the genus field of an Artin--Schreier extension of a  rational congruence function field. In \cite{RockM}, it was studied the particular case of abelian finite $p$-extensions and the explicit description of their genus field was given. That article also studied finite abelian extensions of global rational functions fields, and the term {\it conductor of constants} was introduced for those extensions and it was computed in terms of other invariants of the field.

In \cite{MaRzVi2013, MaRzVi2015} it was developed a theory of genus fields of congruence function fields using Rosen's definition of Hilbert class field. Similarly to the case of numer fields, it was defined the genus field $\g K$ of $K$, as the maximum extension of $K$ such that $K\subseteq\g K\subseteq K_{H,S}$ with $\g K=K k^{\ast}$ and $k^{\ast}/k$ an abelian extension.
When $K/k$ is an abelian extension, $\g K$ is the maximum unramified extension over $K$ such that the primes in $S$ are fully decomposed in $\g K/K$. The methods used there were based on Leopoldt's ideas using Dirichlet characters and gave a general description of $\g K$ in terms of the Dirichlet characters associated to the field $K$. The genus field $\g K$ was obtained for an abelian extension $K$ of $k$ and $S$ the set of the infinite primes.
The method was used to give $\g K$ explicitly when $K/k$ is a cyclic extension of prime degree $\ell\smid q-1$ (Kummer) or cyclic of degree $p$ where $p$ is the characteristic (Artin--Schreier) and also when $K/k$ is a $p$--cyclic extension (Witt). Then, the method was used in \cite{BaRzVi2013} to describe $\g K$ explicitly when $K/k$ is an $\ell^n$-cyclic extension, where $\ell$ is a prime number such that $\ell^n\mid q-1$, $K/k$ is a cyclotomic extension and under a strong restriction.

In this paper, we will use the results obtained in \cite{MaRzVi2013} to describe explicitly the genus field of a cyclic extension of degree $\ell^n$ where $\ell^n\smid q-1$. In this way, we complete what was developed in \cite{BaRzVi2013}. Here we consider a cyclic Kummer extension of degree $\ell^n$, not necessarily cyclotomic and without any restrictions. In case $n=1$ this is the result of Peng. Our method is based on Leopoldt's ideas and therefore they differ from the methods used by Peng which are based on the global function fields analog to the exact hexagon of P. E. Conner and J. Hurrelbrink \cite{Hexagon}. In \cite{MaRzVi2013} the case $n=1$ is described in a little bit different way from how it was originally presented by Peng. Here we will show that using our methods it is possible to give the same description as in the original article.

Finally, given two finite abelian extensions $K_i/k$, $i=1,2$, we have $(K_1)_{\eu{ge}}(K_2)_{\eu{ge}}$ $\subseteq (K_1K_2)_{\eu{ge}}$ but in general there is no equality. This fact is a big obstruction to study the genus field of a finite abelian extension $K/k$ since we cannot just study the components say $K=K_1\dots K_s$. In Section \ref{S5} we study the extension $(K_1K_2)_{\eu{ge}}/(K_1)_{\eu{ge}}(K_2)_{\eu{ge}}$. The main result is that the degree of this extension divides $(q-1)^2$. It is worth to mention that the factor $(q-1)^2$ is due to two different facts. One factor comes from the cyclotomic case where in a composition of two fields such that $\p$ ramifies in both components, in the composition some ramification of $\p$ shifts to descomposition since we have tame ramification. The other factor comes from the various descomposition groups given in the main theorem of the structure of the genus field $\g K=E_\eu{ge}^HK$. As a consequence we obtain that for extensions of degree relatively prime to $q-1$, we have $(K_1)_{\eu{ge}}(K_2)_{\eu{ge}}= (K_1K_2)_{\eu{ge}}$. This was used in \cite{RockM}

\section{Notation}\label{S2}

The notation we will use throughout the paper is the following.

\begin{itemize}
    \item[]$\gcd\quad$ denotes the greatest common divisor;
    \item[]${\rm lcm}\quad$ denotes the least common multiple;
    \item[]$\nu_\ell\quad$ denotes the valuation respect to $\ell$, a prime number;
    \item[]$k\quad$ denotes the rational function field ${\mathbb F}_q(T)$;
    \item[]$R_T\quad$ denotes the polynomial ring ${\mathbb F}_q[T]$;
    \item[]$R_T^+\quad$ denotes the set of monic and irreducible polynomials in $R_T$;
    \item[]$K\quad$ denotes a cyclic finite Kummer extension of $k$; 
    \item[]$\p\quad$ denotes the infinite prime in $k$;
    \item[]$S_\infty(K)\quad$ denotes the set of primes in $K$ over $\p$;
    \item[]$\e{E/F}\quad$ denotes the ramification index of $\p$ in the field extension $E/F$;
    \item[]$\f{E/F}\quad$ denotes the inertia degree of $\p$ in the field extension $E/F$;
    \item[]$\Lambda_N\quad$ denotes the $N$-torsion of the Carlitz's module for $N\in R_T\setminus\{0\}$;
    \item[]$\g K\quad$ denotes the genus field of $K$;
    \item[]$K_{{\eu {gex}}}\quad$ denotes the extended or narrow genus field of $K$;
    \item[]$\lam N\quad$ denotes the cyclotomic function field over $k$ determined by $N$;
    \item[]$\lam N^+\quad$ denotes the maximum real subfield of $\lam N$;
    \item[]$L_n\quad$ denotes the subfield of ${\lam{{T^{-n-1}}}}$ fixed by ${\mathbb F}^*_q$;
    \item[]$K_m\quad$ denotes the field $K{\mathbb F}_{q^m}$;
    \item[]${_n}K\quad$ denotes the field $KL_n$;
    \item[]${_n}K_m\quad$ denotes the field $K{\mathbb F}_{q^m}L_n$.
    \item[]
\end{itemize}

Let $\ell$ be a prime number such that $\ell^n\smid q-1$.
Let $D\in R_T$, $D=P_1^{\alpha_1}\cdots P_r^{\alpha_r}$ with $P_1,\dots,P_r\in R_T^+$ different polynomials, $r\geq 1$, $\alpha_1,\dots,\alpha_r\in\N$ and $1\leq \alpha_j\leq \ell^n-1$, $ 1\leq j\leq r$. Let $E:=k(\sqrt[\ell^n]{(-1)^{\deg D}D})$. We have $E\subseteq \lam{D}$\cite[Corolario 9.5.12]{VilCC}. We denote $D^*:=(-1)^{\deg D}D$, that is $k(\sqrt[\ell^n]{(-1)^{\deg D}D})=k(\sqrt[\ell^n]{D^*})$. Note that if $\ell^n\smid \deg D$, then $k=k(\sqrt[\ell^n]{(-1)^{\deg D}})$ and thus $k(\sqrt[\ell^n]{D^*})=k(\sqrt[\ell^n]{D})$.

Let $\alpha_j=b_j\ell^{a_j}$ with $\gcd(b_j,\ell)=1$ and $\deg P_j=c_j\ell^{d_j}$ with $\gcd(c_j,\ell)=1$, $1\leq j\leq r$. For each $1\leq j\leq r$, we define $E_j:=k(\sqrt[\ell^n]{(P_j^{\alpha_j})^*})$. We have $E_j\subseteq k(\Lambda_{P_j})$, $1\leq j\leq r$. Note that
$\sqrt[\ell^n]{(-1)^{\deg P_j^{\alpha_j}}P^{\alpha_j}}=\sqrt[\ell^n]{(-1)^{b_j\ell^{a_j} \deg P_j}P_j^{b_j\ell^{a_j}}}=\sqrt[\ell^{n-a_j}]{(-1)^{\deg P_j^{b_j}}P_j^{b_j}}$. Further $\sqrt[\ell^{n-a_j}]{(-1)^{\deg P_j^{b_j}}P_j^{b_j}}=(\sqrt[\ell^{n-a_j}]{(-1)^{\deg P_j}P_j})^{b_j}$ $\in k(\sqrt[\ell^{n-a_j}]{(-1)^{\deg P_j}P_j})$. We have that $P_j^{b_j}$ is $\ell$-power free and thus $P_j$ is fully ramified in $k(\sqrt[\ell^{n-a_j}]{(-1)^{\deg P_j^{b_j}}P_j^{b_j}})/k$. Therefore
\begin{align*}
[k(\sqrt[\ell^{n-a_j}]{(-1)^{b_j\deg P_j}P_j^{b_j}}):k]&=\deg(X^{\ell^{n-a_j}}-(-1)^{b_j\deg P_j}P_j^{b_j})\\
&=\deg(X^{\ell^{n-a_j}}-(-1)^{\deg P_j}P_j)\\
&=[k(\sqrt[\ell^{n-a_j}]{(-1)^{\deg P_j}P_j}):k].    
\end{align*}

It follows that $E_j=k(\sqrt[\ell^{n-a_j}]{P_j^*})$. 

Define $M:=E_1\cdots E_r$. We have that $M/k$ is the maximum extension of $E$ unramified on every finite prime and contained in a cyclotomic function field \cite[Proposition 3.3]{MaRzVi2013}. In particular $\g E\subseteq M$. We have that $\p$ is fully ramified in $M/\g E$. 
The field $M$ is known as the {\it narrow genus field} of $E$ and it is denoted by $E_{\eu {gex}}$.
From Abhyankar's Lemma \cite[Theorem 12.4.4]{Vil2006}, we have
\[
e_{\infty}(E_{\eu {gex}}/k)={\rm lcm}[e_{\infty}(E_j/k)\smid 1\leq j\leq r]=:\ell^m.
\] 

Since $\p$ is unramified in the extension $\g E/E$, i.e, $\e{\g E/E}=1$, we have in general $$[E_{\eu {gex}}:\g E]=\e{E_{\eu {gex}}/\g E}\e{\g E/E}=\e{E_{\eu {gex}}/E}.$$

Let $P\in R_T$ and let $F:=k(\sqrt[\ell^n]{(P^{\alpha})^*})$ with $\alpha=b\ell^a$, $\gcd(b,\ell)=1$. We denote by $e_{P}(F/k)$ the  ramification index of the polynomial $P$ in the extension $k(\sqrt[\ell^n]{(P^{\alpha})^*})/k$. Let $e_{\infty}(E/k)$ denote the ramification index of the infinite prime $\p$, in the extension $k(\sqrt[\ell^n]{(P^{\alpha})^*})/k$.

We have the following proposition.

\begin{prop}\label{p1} Let $P\in R_T$ and $k(\sqrt[\ell^n]{(P^{\alpha})^*})=k(\sqrt[\ell^{n-a}]{P^*})$. Then
\begin{gather*}
    e_{P}(k(\sqrt[\ell^n]{(P^{\alpha})^*})/k)=\frac{\ell^n}{\gcd(\alpha,\ell^n)}=\ell^{n-a}
    \quad\textit{ and }\\
    e_{\infty}(k(\sqrt[\ell^n]{(P^{\alpha})^*})/k)
    =\frac{\ell^n}{\gcd(\ell^n,\deg P^\alpha)}
    =\frac{\ell^{n-a}}{\gcd(\ell^{n-a},\deg P)}.
\end{gather*}
\end{prop}
\begin{proof}
See \cite[Subsection 5.2]{MaRzVi2016} or \cite[Teorema 10.3.1]{VilCC}.
\end{proof}

Let $\deg P=c\ell^d$ with $\gcd(c,\ell)=1$. Then $\gcd(\ell^{n-a},\deg P)=\ell^{{\rm min}\{n-a,d\}}$. From Proposition \ref{p1} it follows that
    \begin{gather*}
        e_{\infty}\left( k\left(\sqrt[\ell^{n-a}]{P^*}\right)/k\right)
        =\ell^{n-a-{\rm min}\{n-a,d\}}.
    \end{gather*}
   
Thus 
\begin{gather}\label{divRaM}
    e_{\infty}(E_j/k)=\ell^{n-a_j-{\rm min}\{n-a_j,d_j\}}\smid \ell^m=e_{\infty}(E_{\eu {gex}}/k),\,\, 1\leq j\leq r. 
\end{gather}

Let $K/\F$ be a finite abelian field extension of $k$, where $\p$ is tamely ramified. Let $N\in R_T$ and $u\in{\ma N}$ be such that $\g K{}\subseteq {\lam N}_{u}$.
Let $E_{{\eu {ge}}}$ be the genus field of $E:=\lam N\cap
K_{u}$ and let $E_{{\eu {ge}},u}=E_{{\eu {ge}}}{\ma F}_{q^u}$. 
Let $H$ be the decomposition group of $\S K$, the set of primes in $K$ over $\p$, in $\g E{}K/K$. Let $H'
=H|_{\g E{}}$. 
Then the genus field of $K$ is 
$
K_{{\eu {ge}}}= (\g E{}K)^H=\g E^{H'}K
$
\cite[Theorem 4.2]{MaRzVi2015}.
We have the following diagram
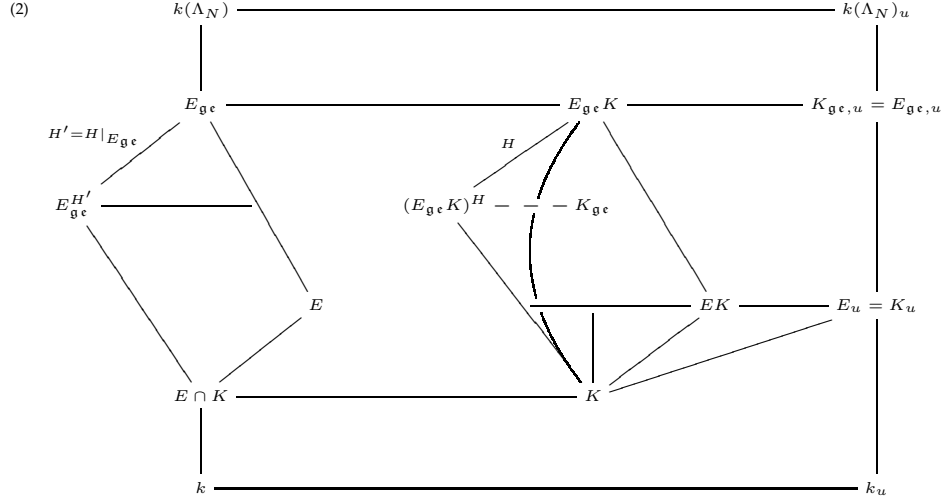
\begin{tiny}
\begin{gather}\label{diag}
\xymatrix{
&k(\Lambda_N)\ar@{-}[rrrrr]\ar@{-}[d]&&&&&k(\Lambda_N)_u\ar@{-}[d]\\
&E_{{\eu {ge}}}\ar@{-}[dl]_{H'=H|_{\g E{}}}\ar@{-}[ddr]\ar@{-}[rrr]&&&
E_{{\eu {ge}}}K\ar@{-}[ddr]\ar@{-}[dl]_H\ar@{-}[rr]\ar@/_2pc/@{-}[ddd]
|!{[dl];[d]}\hole|!{[ddll];[ddr]}\hole
&&K_{{\eu {ge}},u}=E_{{\eu {ge}},u}\ar@{-}[dd]\\
\g E^{H'}\ar@{-}[ddr]\ar@{-}[rrr]|!{[ur];[drr]}\hole
&&&(\g E{}K)^H\ar@{-}[ddr]\ar@{--}[r]&\g K{}
\ar@{-}[dd]|!{[dll];[dr]}\hole\\
&&E\ar@{-}[rrr]|!{[ru];[rrd]}\hole \ar@{-}[dl]&&&EK\ar@{-}[r]\ar@{-}[dl]
&E_u=K_u\ar@{-}[dd]\ar@{-}[dll]\\
&E\cap K\ar@{-}[rrr]\ar@{-}[d]&&&K\\
&k\ar@{-}[rrrrr]&&&&&k_u}
\end{gather}
\end{tiny}

\section{Cyclotomic $\ell^n$ case}\label{S3}
 
We start with the following proposition that bounds the type of ramification in Kummer $\ell^n$-extensions.

\begin{prop}\label{imenosj} Let $P,Q\in R_T^+$ and let $J:=k(\sqrt[\ell^n]{(P^\alpha)^*})$ and $F:=k(\sqrt[\ell^n]{(Q^\beta)^*})$. Suppose that $e_P(J/k)\leq e_Q(F/k)$ and $e_{\infty}(F/k)\leq e_{\infty}(J/k)$ with $1<e_{\infty}(J/k)$. Then $\nu_\ell(\deg P)\leq\nu_\ell(\deg Q)$.
\end{prop}
\begin{proof}
Since $e_P(J/k)\leq e_Q(F/k)$, we have $1\leq\frac{e_Q(F/k)}{e_P(J/k)}$. On the other hand, from $e_{\infty}(F/k)\leq e_{\infty}(J/k)$ we obtain
\begin{gather*}
    \frac{e_Q(F/k)}{\gcd(\deg Q, e_Q(F/k))}=e_{\infty}(F/k)\leq 
    e_{\infty}(J/k)=\frac{e_P(J/k)}{\gcd(\deg P, e_P(J/k))}.
\end{gather*}
Since $e_{\infty}(J/k)\neq 1$, it follows that $\nu_\ell(\deg P)< e_P(J/k)$, and $\gcd(\deg P,e_P(J/k))=\ell^{\nu_\ell(\deg P)}$. Also, since $\gcd(\deg Q, e_Q(F/k))\smid \deg Q$, we obtain
\begin{gather*}
    \frac{e_Q(F/k)}{\ell^{\nu_\ell(\deg Q)}}\leq\frac{e_Q(F/k)}{\gcd(\deg Q, e_Q(F/k))}\leq \frac{e_P(J/k)}{\gcd(\deg P, e_P(J/k))}=\frac{e_P(J/k)}{\ell^{\nu_\ell(\deg P)}}.
\end{gather*}

Hence
$$1\leq \frac{e_Q(F/k)}{e_P(J/k)}\leq \ell^{\nu_\ell(\deg Q)-\nu_\ell(\deg P)}.$$ Therefore $0\leq \nu_\ell(\deg Q)-\nu_\ell(\deg P)$, i.e, $\nu_\ell(\deg P)\leq\nu_\ell(\deg Q)$.
\end{proof}

The main result for the Kummer cyclotomic $\ell^n$-cyclic case is the next theorem.

\begin{teo}\label{Pengln} Let $E =k(\sqrt[\ell^n]{D^*})$, with $D=P_1^{\alpha_1}\cdots P_r^{\alpha_r}$, $1\leq \alpha_j\leq \ell^n-1$, $\alpha_j=b_j\ell^{a_j}$ with $\gcd(b_j,\ell)=1$, $1\leq j \leq r$, $P_1,\dots, P_r\in R_T^+$ different with $\deg P_j=c_j\ell^{d_j}$, ${\rm gcd}(c_j,\ell)=1$, $1\leq j\leq r$. We order the polynomials $P_1,\dots,P_r$ such that $0=a_1\leq\cdots\leq a_r\leq n-1$.

Let $E_{\eu {gex}}:=E_1\cdots E_r$ with $E_j=k(\sqrt[\ell^{n-a_j}]{P_j^*})$, $1\leq j\leq r$. Let
\begin{align*}
    e_\infty(E/k)=\ell^t\;\textit{with }\;t&=n-{\rm min}\{n,\nu_\ell(\deg D)\},\\
    &\\
    e_\infty(E_{\eu {gex}}/k)=\ell^m\;\textit{with }\;m&=\underset{1\leq j\leq r}{\rm max}\nu_\ell(e_\infty(E_j/k))\\
    &={\rm max}\{n-a_j-{\rm min}\{n-a_j,d_j\} \smid 1\leq j\leq r\}.
\end{align*}

Let $i_0$, $1\leq i_0\leq r$, be such that $n-a_{i_0}-{\rm min}\{n-a_{i_0},d_{i_0}\}=m$ and $n-a_j-d_j<m$ for $j>{i_0}$.
For $m>0$ we have $\gcd(\deg P_{i_0},\ell^n)=\ell^{d_{i_0}}$, and therefore there exist $a, b\in\Z$ such that $a\,\deg P_{i_0}+b\ell^n=\ell^{d_{i_0}}$. For $j<{i_0}$, we have $d_{i_0}\leq d_j$. Let $z_j:=-ac_j\ell^{d_j-d_{i_0}}$. For $j>{i_0}$, let $y_j\equiv -c_jc_{i_0}^{-1}\mod \ell^{n}\in\Z$.

Then $$\g E=F_1\cdots F_r,$$ where $F_j=E_j$ with $1\leq j\leq r$ if $m=t$, i.e, $\g E=E_{\eu {gex}}$, and if $m>t\geq 0$, then
\begin{gather}\label{knorr}
    F_j:=
    \begin{cases}
    k\left(\sqrt[\ell^{n-a_j}]{P_jP_{i_0}^{z_j}}\right)&\text{if $j<{i_0}$,}\\
    k\left(\sqrt[\ell^{d_{i_0}+t}]{P_{i_0}^*}\right)&\text{if $j={i_0}$,}\\
    k\left(\sqrt[\ell^{n-a_j}]{P_jP_{i_0}^{y_j\ell^{d_j-d_{i_0}}}}\right)&\text{if $j>{i_0}$ and $d_j\geq d_{i_0}$},\\
    k\left(\sqrt[\ell^{n-a_j+d_{i_0}-d_j}]{P_j^{\ell^{d_{i_0}-d_j}}P_{i_0}^{y_j}}\right)&\text{if $j>{i_0}$ and $d_{i_0}> d_j$.}
    \end{cases}
\end{gather}
\end{teo}
\begin{proof}
First suppose that $m=t$. Then $$[E_{\eu {gex}}:\g E]=\frac{\e{E_{\eu {gex}}/k}}{\e{E/k}}=\ell^{m-t}=1.$$ It follows that $$E_{\eu {gex}}=\g E.$$
     
Now, suppose that $m>t$. Let $i_0$ be as before. From (\ref{divRaM}) we have
$$e_{\infty}\left( k\left(\sqrt[\ell^{n-a_{i_0}}]{P_{i_0}^*}\right)/k\right)=\eu{\ell^{n-a_{i_0}}}{\gcd(\deg P_{i_0}, \ell^{n-a_{i_0}})}=\ell^m.$$ 
    
Therefore, $\gcd(\deg P_{i_0}, \ell^{n-a_{i_0}})=\ell^{d_{i_0}}$ and $m=n-a_{i_0}-{\rm min}\{n-a_{i_0}, d_{i_0}\}=n-a_{i_0}-d_{i_0}>t\geq0$.
    
Because $\gcd(\deg P_{i_0},\ell^n)=\ell^{d_{i_0}}$, there exist $a,b\in\Z$, such that
\begin{equation}\label{combilin}
        a\deg P_{i_0}+b\ell^n=\ell^{d_{i_0}}.
\end{equation}

In particular $ac_{i_0}+b\ell^{n-d_{i_0}}=1$ and therefore $\gcd(a,\ell)=1$.
From Proposition \ref{imenosj} we have $\ell^{d_{i_0}}\smid \deg P_j$ for $1\leq j\leq {i_0}-1$, so that from (\ref{combilin}) we obtain
    
$$\deg P_j+\left(-a\frac{\deg P_j}{\ell^{d_{i_0}}}\right)\deg P_{i_0}=(b\,\deg P_j)\ell^{n-d_{i_0}}.$$
    
We have $-a\frac{\deg P_j}{\ell^{d_{i_0}}}\in \Z$ with $\gcd(a,\ell)=1$. Let $z_j:=-a\frac{\deg P_j}{\ell^{d_{i_0}}}=-ac_j\ell^{d_j-d_{i_0}}$ and $Q_j:=P_jP_{i_0}^{z_j}$, $1\leq j\leq {i_0}-1$. By construction we have that $\ell^n\smid \deg Q_j$, since $d_j\geq d_{i_0}$. Therefore, $\p$ is unramified at $k\left(\sqrt[\ell^{n-a_j}]{Q_j}\right)/k$, i.e, $e_{\p}\left(k\left(\sqrt[\ell^{n-a_j}]{Q_j}\right)/k\right)=1$, (Proposition \ref{p1}). 
Also, for $j<{i_0}$, we have  
\begin{align}
        e_{P_j}\left(k\left(\sqrt[\ell^{n-a_j}]{Q_j}\right)/k\right)&=\ell^{n-a_j},\label{I}\\
        e_{P_{i_0}}\left(k\left(\sqrt[\ell^{n-a_j}]{Q_j}\right)/k\right)&=\frac{\ell^{n-a_j}}{\gcd(z_j,\ell^{n-a_j})}
        = \ell^{n-a_j-{\rm min}\{n-a_j, d_j-d_{i_0}\}}.\nonumber
\end{align}

 Now $d_j-d_{i_0}\geq {\rm min}\{n-a_j,d_j-d_{i_0}\}$ so that $n-a_j-{\rm min}\{n-a_j,d_j-d_{i_0}\}\leq n-a_j-d_j+d_{i_0}\leq n-a_{i_0}-d_{i_0}+d_{i_0}=n-a_{i_0}$ since $e_\infty(E_j/k)=\ell^{n-a_j-d_j}\smid e_\infty(E_{i_0}/k)=\ell^{n-a_{i_0}-d_{i_0}}$. Therefore
 \begin{equation}
        e_{P_{i_0}}\left(k\left(\sqrt[\ell^{n-a_j}]{Q_j}\right)/k\right)\leq \ell^{n-a_{i_0}}.\label{II}
\end{equation}

Now consider $j>{i_0}$. We have two possibilities: $d_{i_0}\leq d_j$ or $d_{i_0}>d_j$.
 
First suppose that $d_{i_0}\leq d_j$. Consider the polynomial $Q_j:=P_jP_{i_0}^{x_j}$, with $x_j:=y_j\ell^{d_j-d_{i_0}}$. Hence
\begin{align*}
        \deg Q_j&=\deg P_j+x_j\deg P_{i_0}
        =c_j\ell^{d_j}+y_jc_{i_0}\ell^{d_j}.
\end{align*}

Then $\deg Q_j=\ell^{d_j}(c_j+y_jc_{i_0})=A\ell^B$ with $\gcd(A,\ell)=1$. Note that we may choose $B\geq n+d_j$, since this is equivalent to $c_j+y_jc_{i_0}\equiv 0\mod \ell^{n}$, and this is possible because $\gcd(c_{i_0}c_j,\ell)=1$. 
   
Finally, since $\gcd(c_{i_0}c_j,\ell)=1$, we select $y_j\in\Z$ such that $y_j\equiv -c_jc_{i_0}^{-1}\mod \ell^{n}$. Note that in particular, ${\rm gcd}(y_j,\ell)=1$. Thus, define $F_j:=k\left(\sqrt[\ell^{n-a_j}]{Q_j}\right)$. We have $\deg Q_j=A\ell^B$ with $B\geq n$. Note that
\begin{gather}
        e_{P_j}(F_j/k)=e_{P_j}(E_j/k)=\ell^{n-a_j},\label{III}\\ 
\begin{align}        e_{P_{i_0}}(F_j/k)&=\frac{\ell^{n-a_j}}{\gcd(y_j\ell^{d_j-d_{i_0}},\ell^{n-a_j})}=\ell^{n-a_j-{\rm min}\{n-a_j,d_j-d_{i_0}\}}\nonumber\\
&\leq \ell^{n-a_{i_0}}=e_{P_{i_0}}(E_{i_0}/k)
\end{align}\label{IV}\\
{\rm{and}}\quad e_{\infty}\left(F_j/k\right)=\frac{\ell^{n-a_j}}{\gcd(\deg Q_j,\ell^{n-a_j})}=\ell^{n-a_j-n+a_j}=1.\label{V}
\end{gather}

Now, suppose that $d_{i_0}>d_j$. 
Let $Q'_j:=P_j^{\ell^{d_{i_0}-d_j}}P_{i_0}^{y_j}$, with $y_j$ as before, that is, $y_j\equiv -c_jc_{i_0}^{-1}\mod \ell^{n}$, $Q'_j\in R_T$. Define $$F_j:=k\left(\sqrt[\ell^{n-a_j+d_{i_0}-d_j}]{P_j^{\ell^{d_{i_0}-d_j}}P_{i_0}^{y_j}}\right).$$
    
We have 
\begin{align}\label{JGi}
       e_{P_j}(F_j/k)&=\frac{\ell^{n-a_j+d_{i_0}-d_j}}{\gcd(\ell^{d_{i_0}-d_j},\ell^{n-a_j+d_{i_0}-d_j})}
       =\ell^{n-a_j+d_{i_0}-d_j-d_{i_0}+d_j}\nonumber\\
       &=\ell^{n-a_j}=e_{P_j}(E_j/k),
\end{align}
and
\begin{gather*}
       e_{P_{i_0}}(F_j/k)=\frac{\ell^{n-a_j+d_{i_0}-d_j}}{\gcd(y_j,\ell^{n-a_j+d_{i_0}-d_j})}
       =\ell^{n-a_j+d_{i_0}-d_j}.
\end{gather*}

Since $n-a_j-d_j<m=n-a_{i_0}-d_{i_0}$, then $n-a_j-d_j+d_{i_0}<n-a_{i_0}$. It follows that
\begin{equation}
        e_{P_{i_0}}(F_j/k)<\ell^{n-a_{i_0}}=e_{P_{i_0}}(E_{i_0}/k).\label{VI}
\end{equation}

Finally 
\begin{align*}
       \deg Q'_j&=\deg P_j^{\ell^{d_{i_0}-d_j}}P_{i_0}^{y_j}=\ell^{d_{i_0}-d_j}\deg P_j+y_j\deg P_{i_0}\\
       &=\ell^{d_{i_0}-d_j}c_j\ell^{d_j}+y_jc_{i_0}\ell^{d_{i_0}}=\ell^{d_{i_0}}(c_j+y_jc_{i_0})=\ell^{d_{i_0}}(A\ell^{n-d_j})
       =A\ell^{n+d_{i_0}-d_j}.
\end{align*}

We have $A\ell^{n+d_{i_0}-d_j}=A\ell^{n-d_j+d_{i_0}}$, with $\nu_\ell(A)=0$ and $n-d_j+d_{i_0}>n$. It follows that  
\begin{equation}
        e_{\infty}(F_j/k)=\frac{\ell^{n-a_j+d_{i_0}-d_j}}{\gcd(A\ell^{n+d_{i_0}-d_j},\ell^{n-a_j+d_{i_0}-d_j})}=\ell^{n-a_j+d_{i_0}-d_j-n+a_j-d_{i_0}+d_j}=1.\label{VII}
\end{equation}

Let $L:=F_1\cdots F_{{i_0}-1}F_{{i_0}+1}\cdots F_r$. From (\ref{I}) -- (\ref{VII}) we have $L\subseteq \g E$. We will prove that $\g E=LF_{i_0}$, with $F_{i_0}$ as in (\ref{knorr}). Consider the sets 
\[
    \mathcal J=\{j\in\{1,2,\dots,r\}\smid j>{i_0},\;n-a_j-d_j>t\;\textit{ and }\;d_{i_0}> d_j\}
\]
and $\I=\{1,2,\dots,r\}\setminus (\J\cup \{{i_0}\})$. We order the elements of $\J=\{j_1,\dots,j_s\}$ so that $d_{i_0}-d_{j_1}\leq d_{i_0}-d_{j_2}\leq\cdots\leq d_{i_0}-d_{j_s}$. For $j_u\in\J$ we have $F_{j_u}=k(\sqrt[\ell^{n-a_{j_u}+d_{i_0}-d_{j_u}}]{P_{j_u}^{\ell^{d_{i_0}-d_{j_u}}}P_{i_0}^{y_{j_u}}})$.
    
Let $I_{j}$ be the inertia group of the prime $P_{j}$ in $F_{j}/k$ for $1\leq j\leq r$. We have $|I_{j}|=e_{P_j}(F_j/k)=\ell^{n-a_j}$, $j\neq i_0$. Let $F'_{j}:=F_{j}^{I_{j}}$. If $j_u\in\J$ we have $F'_{j_u}=k(\sqrt[\ell^{d_{i_0}-d_{j_u}}]{P_{i_0}^*})$, $1\leq u\leq s$, and $F'_{j_1}\subseteq F'_{j_2}\subseteq \cdots\subseteq F'_{j_s}$. If $j\in\I$ we have $F'_j=k$.
    
For $F_i$ and $F_j$ given as in (\ref{knorr}) with $i\neq j$, we obtain $$\Gal{(F_iF_j/F_i\cap F_j)}\cong \Gal{(F_i/F_i\cap F_j)}\times\Gal{(F_j/F_i\cap F_j)}.$$
    
In addition, we have 
\begin{gather*}
        |\Gal{(F_i/F'_i)}|=|\Gal{(F_i/F_i^{I_i})}|=|I_i|=e_{P_i}(F_i/k)=\ell^{n-a_i},\textit{ and}\\
        |\Gal{(F_j/F'_j)}|=|\Gal{(F_j/F_j^{I_j})}|=|I_j|=e_{P_j}(F_j/k)=\ell^{n-a_j}.
\end{gather*}

Hence $\Gal{(F_iF_j/F_i\cap F_j)}\cong I_i\times I_j.$
    
Now, from Abhyankar's Lemma \cite[Theorem 12.4.4]{Vil2006}, we have
\[
    e_{P_i}(F_iF_j/k)={\rm lcm}[e_{P_i}(F_i/k),e_{P_i}(F_j/k)]={\rm lcm}[e_{P_i}(F_i/k),1]=e_{P_i}(F_i/k)=|I_i|.
\]

Similarly we have $e_{P_j}(F_iF_j/k)=e_{P_i}(F_j/k)=|I_j|$. Let $I'_i$ and $I'_j$ be the inertia groups of $P_i$ and $P_j$ in $F_iF_j/k$ respectively. Then $I_i\times\{e\}<I'_i$ and $\{e\}\times I_j<I'_j$ are such that $|I_i\times\{e\}|=|I'_i|$ and $|\{e\}\times I_j|=|I'_j|$. Therefore, the maximum unramified field subextension at $P_i$ and $P_j$ of $F_iF_j/k$ is $(F_iF_j)^{I_iI_j}=(F_iF_j)^{I_i\times I_j}$.
    
Also note that for  $i,j\in\J$, $i\neq j$, we have that $P_i$ and $P_j$ are unramified in $F_i\cap F_j$. 
Hence $F_i\cap F_j\subseteq F'_i$. Similarly $F_i\cap F_j\subseteq F'_j$. Thus $F_i\cap F_j\subseteq F'_i\cap F'_j\subseteq F_i\cap F_j$. Therefore $F_i\cap F_j= F'_i\cap F'_j$. If we assume that $i<j$, then $F_i\cap F_j=F'_i$. Therefore
\begin{align*}
    [F_iF_j:k]&=[F_iF_j:F_i\cap F_j][F_i\cap F_j:k]=[F_iF_j:F'_i][F'_i:k]\\
    &=[F_iF_j:F'_j][F'_j:F'_i][F'_i:k]=[F_iF_j:F_j][F_j:F'_j][F'_j:k]\\
    &=[F_i:F'_i][F_j:F'_j][F'_j:k].
\end{align*}

Now, say that $\I=\{i_1,\dots,i_{s'}\}$. Using induction it follows that the fields $F_{i_w}$, with $i_w\in\I$, for $1\leq w\leq s'$ satisfy
\begin{itemize}
        \item[a.1.-] $F_{i_1}\cdots F_{i_{w-1}}\cap F_{i_w}=k$,
        \item[b.1.-] $I_{i_1}\cdots I_{i_w}\cong I_{i_1}\times\cdots\times I_{i_w}$,
        \item[c.1.-] $[F_{i_1}\cdots F_{i_w}:k]=\prod\limits_{j=1}^w[F_{i_j}:k]=\prod\limits_{j=1}^w \ell^{n-a_{i_j}}$, 
        \item[d.1.-] $(F_{i_1}\cdots F_{i_w})^{I_{i_1}\cdots I_{i_w}}=k$.
\end{itemize}
    
Similarly, the fields $F_{j_w}$, $1\leq w\leq s$, with $j_w\in\J$, satisfy
\begin{itemize}
    \item[a.2.-] $F_{j_1}\cdots F_{j_{w-1}}\cap F_{j_w}=F'_{j_{w-1}}$,
    \item[b.2.-] $I_{j_1}\cdots I_{j_w}\cong I_{j_1}\times\cdots\times I_{j_w}$,
    \item[c.2.-] $[F_{j_1}\cdots F_{j_w}:k]=(\prod\limits_{n=1}^{w-1}[F_{j_n}:F'_{j_n}])[F'_w:k]=(\prod\limits_{n=1}^{w-1} \ell^{n-a_{j_n}})\ell^{d_i-d_w}$,
    \item[d.2.-] $(F_{j_1}\cdots F_{j_w})^{I_{j_1}\cdots I_{j_w}}=F'_{j_w}$.
\end{itemize}

Let $L=\prod\limits_{i\in\I} F_i \prod\limits_{j\in\J} F_j$. We will see that $(\prod\limits_{i\in\I}F_i)\cap (\prod\limits_{j\in\J} F_j)=k$. Otherwise, let $A\neq k$ be a proper subfield of $\prod\limits_{i\in\I}F_i$. Then at least one $P_i$ with $i\in \I$ should ramify in $A$, since otherwise we would have from (b.1) and (d.1) that $A\subseteq (\prod\limits_{i\in\I}F_i)^{\prod\limits_{i\in\I}I_i}=k$. Therefore, in every nontrivial proper subfield of $\prod\limits_{i\in\I}F_i$ at least one $P_i$ with $i\in\I$ is ramified. Now, in every subfield of $\prod\limits_{j\in\J} F_j$ some $P_j, j\in \mathcal J$ or $P_{i_0}$ is ramified but none of these is ramified in $\prod\limits_{i\in\I} F_i$. It follows that $(\prod\limits_{i\in\I}F_i)\cap (\prod\limits_{j\in\J} F_j)$ does not have proper subfields distinct of $k$. As consequence of (c.1) and (c.2) we obtain
\begin{equation}\label{nueve}
\begin{split}
    [L:k]&=[F_{i_1}\cdots F_{i_{s'}}:k][F_{j_1}\cdots F_{j_s}:k]\\&=\prod\limits_{j=1}^{s'} \ell^{n-a_{i_j}}(\prod\limits_{n=1}^{s-1} \ell^{n-a_{j_n}})\ell^{d_{i_0}-d_s}=(\prod\limits_{\stackrel{j=1}{j\neq {i_0}}}^{r} \ell^{n-a_{j}})\ell^{d_{i_0}-d_s}.
\end{split}
\end{equation}

Next we will see that $L\cap F_{i_0}=F'_s$. Let $C:=L\cap F_{i_0}$. We have $F'_s\subseteq C$. Since $C\subseteq F_{i_0}$, every prime $P_j$ with $1\leq j\leq r$, $j\neq {i_0}$, is unramified in $C$. Let $I=\prod\limits_{\stackrel{j=1}{j\neq {i_0}}}^r I_j$. From (b.1) and (b.2) we obtain $C\subseteq L^I$. From (d.1) and (d.2) it follows
\[
    L^I=(F_{i_1}\cdots F_{i_{s'}})^{I_{i_1}\cdots I_{i_{s'}}}(F_{j_1}\cdots F_{j_s})^{I_{j_1}\cdots I_{j_s}}=F'_s.
\]
    
Therefore $C=L^I=F'_s$.
    
Finally, note that $[M:\g E]=[M:LF_{i_0}]$. By the Galois correspondence we have $[LF_{i_0}:L]=[F_{i_0}:L\cap F_{i_0}]$. Thus 
\begin{align*}
    [F_{i_0}:L\cap F_{i_0}]&=[k(\sqrt[\ell^{d_{i_0}+t}]{P_{i_0}}):k(\sqrt[\ell^{d_{i_0}-d_s}]{P_{i_0}^*})]\\
    &=\frac{[k(\sqrt[\ell^{d_{i_0}+t}]{P_{i_0}}):k]}{[k(\sqrt[\ell^{d_{i_0}-d_s}]{P_{i_0}^*}):k]}=\ell^{d_{i_0}+t-d_{i_0}+d_s}.
\end{align*}

It follows that $[LF_{i_0}:L]=[F_{i_0}:L\cap F_{i_0}]=[F_{i_0}:F'_{i_0}]=\ell^{d_{i_0}+t-d_{i_0}+d_s}=\ell^{d_s+t}$. From equation (\ref{nueve}) we obtain 
\begin{align*}
    [M:LF_{i_0}]&=\frac{[M:k]}{[LF_{i_0}:k]}=\frac{[M:k]}{[LF_{i_0}:L][L:k]}=\frac{\prod\limits_{j=1}^{r} \ell^{n-a_{j}}}{\ell^{d_s+t}(\prod\limits_{\stackrel{j=1}{j\neq {i_0}}}^{r} \ell^{n-a_{j}})\ell^{d_{i_0}-d_s}}
    \\
    &=\frac{\ell^{n-a_{i_0}}}{\ell^{d_{i_0}+t}}=\ell^{n-a_{i_0}-d_{i_0}-t}.
\end{align*}
    
On the other hand, we have $m=n-a_{i_0}-d_{i_0}$. Therefore $[M:LF_{i_0}]=\ell^{m-t}=[M:\g E]$ and by construction we have $L(\sqrt[\ell^{d_{i_0}+t}]{P_{i_0}^*})\subseteq \g E$. It follows that $\g E=L(\sqrt[\ell^{d_{i_0}+t}]{P_{i_0}^*})$.
\end{proof}

\begin{no} {\rm{In the definition of $F_j$ given in \eqref{knorr}, for $j<i_0$ if $1\leq n-a_j-{\rm min}\{n-a_j, d_j\}\leq t$, it can be defined simply $F_j=E_j$.
}}
\end{no}

\section{General $\ell^n$ case}\label{S4}

Recall that $H$ is the decomposition group of $\p$ in $\g EK/K$ (see diagram (\ref{diag})).

\begin{teo}\label{Tprin} Let $K =k(\sqrt[\ell^n]{\gamma D})\subseteq \lam D_u$, with $\gamma\in \F^*$, $D=P_1^{\alpha_1}\cdots P_r^{\alpha_r}$, $1\leq \alpha_j\leq \ell^n-1$, $\alpha_j=b_j\ell^{a_j}$ with $\gcd(b_j,\ell)=1$, $1\leq j \leq r$, $P_1,\dots, P_r\in R_T^+$ different polynomials. We order the polynomials $P_1,\dots,P_r$ so that $0=a_1\leq\cdots\leq a_r\leq n-1$. Let $E=K_u\cap\lam D$, $t$ as in Theorem \ref{Pengln} and $\alpha=\nu_{\ell}(|H|)$. Let $H':=H\smid_{\g E}$. Then $E_{\eu{ge}}^{H'}=F_1\cdots F_{{i_0}-1}F_{{i_0}+1}\cdots F_r (\sqrt[\ell^{d_{i_0}+(t-\alpha)}]{P_{i_0}^*})$ where $F_j$ are given in (\ref{knorr}) for all $j$. Thus
\[
 \g K=E_{\eu{ge}}^{H'}K=\prod_{\stackrel{i=1}{i\neq i_0}}^r F_i K(\sqrt[\ell^{d_{i_0}+(t-\alpha)}]{P_{i_0}^*}).
\]

Further, if $d={\rm{min}}\{n,\nu_\ell(\deg D)\}$, we have
\[
|H|=\ell^\alpha=[\F(\sqrt[\ell^n]{(-1)^{\deg D}\gamma}):\F(\sqrt[\ell^d]{(-1)^{\deg D}\gamma})].
\]
\end{teo}
\begin{proof}
From Theorem \ref{Pengln} we have $\e{F_j/k}=1$ for $j\neq i_0$ (i.e, $\e{L/k}=1$). Therefore $e_{\infty}(E_{\eu {ge}}/k)={\rm lcm}[e_{\infty}(F_j/k)\smid 1\leq j\leq r]=e_{\infty}(F_{i_0}/k)=\ell^t.$ That is, the ramification of $\p$ in $\g E/k$ depends only on the ramification of $\p$ in the extension $F_{i_0}/k$. Since $\g E/\g E^+$ is a cyclic extension and $[\g E:\g E^{H'}]=[\g E:F_1\cdots F_{{i_0}-1}F_{{i_0}+1}\cdots F_r (\sqrt[\ell^{d_{i_0}+(t-\alpha)}]{P_{i_0}^*})]$ it follows that 
\[
\g E^{H'}=F_1\cdots F_{{i_0}-1}F_{{i_0}+1}\cdots F_r (\sqrt[\ell^{d_{i_0}+(t-\alpha)}]{P_{i_0}^*}).
\]

We have $EK=K(\sqrt[\ell^n]{(-1)^{\deg D}\gamma})$ and $EK/K$ is unramified, in fact, $EK/K$ is a constant extension \cite[Subsection 5.3]{MaRzVi2016}. From \cite[Theorem 6.2.1]{Vil2006}, we have $$\f{EK/k}=[\F(\sqrt[\ell^n]{(-1)^{\deg D}\gamma}):\F]$$ and from \cite[Proposition 2.8]{BaRzVi2013} we have $$\f{K/k}=[\F(\sqrt[\ell^{d}]{(-1)^{\deg D}\gamma}):\F] \;\;{\rm with}\;\; d={\rm{min}}\{n,\nu_\ell(\deg D)\}.$$ Then $\f{EK/k}=\f{EK/K}\f{K/k}$ where $\f{EK/K}=|H|=\ell^\alpha$. 
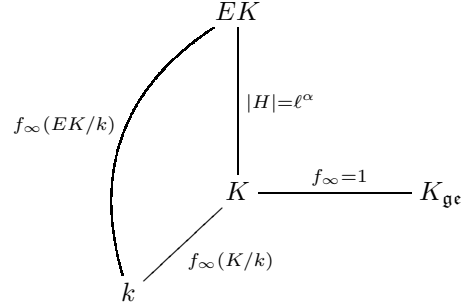
$$\xymatrix{
&EK\ar@{-}[dd]^{|H|=\ell^\alpha}& \\
&&\\
& K \ar@{-}[dl]^{\f{K/k}} &&\g K\ar@{-}[ll]_{f_\infty=1} \\
k\ar@/^2pc/@{-}[uuur]^{\f{EK/k}}&&&
}$$

Hence 
\begin{align*}
    |H|&=\ell^\alpha=\frac{\f{EK/k}}{\f{K/k}}=\frac{[\F(\sqrt[\ell^n]{(-1)^{\deg D}\gamma}):\F]}{[\F(\sqrt[\ell^d]{(-1)^{\deg D}\gamma}):\F]}\\
&=[\F(\sqrt[\ell^n]{(-1)^{\deg D}\gamma}):\F(\sqrt[\ell^d]{(-1)^{\deg D}\gamma})].
\end{align*}
\end{proof}

\begin{cor}[Case $n=1$, G. Peng  \cite{MaRzVi2013}]\label{T5.1.6}
Let $K:=k(\sqrt[\ell]{\gamma D})$ with $\gamma\in\F^*$, where $D=P_1^{\alpha_1}\cdots P_r^{\alpha_r}\in R_T$ is a monic $\ell$--power free polynomial, $P_i\in R_T^+$ and $d_j=\nu_\ell(\deg P_j)$, $1\leq \alpha_j\leq \ell-1$, $1\leq j\leq r$. Let $m={\rm max}\{1-{\rm min}\{1,d_j\} \smid 1\leq j\leq r\}$ and $d={\rm{min}}\{1,\nu_\ell(\deg D)\}$. We assume that $m=1-{\rm {min}}\{1,d_r\}$.
Then
\begin{gather}\label{pengnew}
    K_{{\eu {ge}}}=
    \begin{cases}
    E_1\cdots E_rK&\text{if $m=1-d$, $E=K$ or $E\neq K$ and $d=1$}\\
    F_1\cdots F_{r-1}K&\text{if $m>1-d$, $E=K$ or $E\neq K$ and $d=0$ or $1$}\\
    \end{cases},
\end{gather}
where $E_j=k(\sqrt[\ell]{P_j^*})$, $1\leq j\leq r$, $F_j=k(\sqrt[\ell]{P_jP_{r}^{z_j}})$ $1\leq j\leq r-1$, $z_j=-a\deg P_j$, with $a\deg P_r + b\ell=1$, when $d_r=0$ for some $b$.
\end{cor}
\begin{proof}
Writing $\alpha_j=b_j\ell^{a_j}$, $1\leq j\leq r$, we have $0=a_1=\cdots=a_r$.
Let $i_0$ be as in Theorem \ref{Tprin}, that is, $i_0$, $1\leq i_0\leq r$, is such that $n-a_{i_0}-{\rm min}\{n-a_{i_0},d_{i_0}\}=m$. Then $i_0=r$.

On the other hand we have two cases. If we write $\epsilon:=(-1)^{\deg D}\gamma$, then
\begin{itemize}
    \item[$a)$\,] $\epsilon\in (\F^*)^\ell$,    
    \item[$b)$\,] $\epsilon\notin (\F^*)^\ell$.
\end{itemize}

{\large $a)$} If $K=E$ (see \cite[Theorem 4.2]{MaRzVi2015}). Therefore $\g K=\g E$.  
If $m=t$, then $\g E=E_1\cdots E_r$, with $E_j=k(\sqrt[\ell]{P_j^*})$, $1\leq j\leq r$, so that $\g K=E_1\cdots E_r$. If $m>t$, we have $m=1$ and $t=0$ and from Theorem $\ref{Pengln}$ we have $\g E=F_1\cdots F_r$, where if $z_j:=-ac_j\ell^{d_j-d_{r}}$ then
\begin{gather}\label{knorPeng}
    F_j:=
    \begin{cases}
    k\left(\sqrt[\ell]{P_jP_{r}^{z_j}}\right)&\text{if $j<{r}$,}\\
    k\left(\sqrt[\ell^{d_{r}+t}]{(-1)^{\deg P_r}P_{r}}\right)&\text{if $j={r}$}.
    \end{cases}
\end{gather}
    
Note that $d_r=0$ and since $t=0$, we obtain $F_r=k$. Hence $\g K=F_1\cdots F_{r-1}$.

{\large $b)$} We have that $K\neq E$ and that $\g K=E_{\eu{ge}}^{H'}K$ (where $H$ is as in Theorem \ref{Tprin}), with $\g E$ as in case $a)$. We also have $|H'|=\ell^\alpha=[\F(\sqrt[\ell]{(-1)^{\deg D}\gamma}):\F(\sqrt[\ell^d]{(-1)^{\deg D}\gamma})]$, where $\alpha= 0$ or $1$, that is $|H'|=1$ or $\ell$.

If $d=1$, then $|H'|=1$, that is $\alpha=0$. Thus $\g K=\g EK$, with $\g E$ as in $a)$. Therefore $\g K=E_1\cdots E_rK$ if $m=t$ and $\g K=F_1\cdots F_{r-1}K$ if $m>t$ with $F_j$ as in (\ref{knorPeng}).

If $d=0$, we have $|H'|=\ell$, that is $\alpha=1$. Since $d=1-t$, we have $1=t\leq m\leq 1$. Therefore $\ell\nmid \deg P_r$, that is $d_r=0$. From Theorem \ref{Tprin} and from $d_r+t-\alpha=0$, we obtain
\[
\g E^{H'}=F_1\cdots F_{r-1} (\sqrt[\ell^{d_{i_0}+(t-\alpha)}]{P_{r}^*})=F_1\cdots F_{r-1}.
\]

Therefore $\g K=F_1\cdots F_{r-1}K$.
\end{proof}

\begin{ejm} {\rm{Let $K=\F(\sqrt[\ell^n]{\gamma D})$ with $\gamma=5$, $\ell = 3$, $n=10$, $q=472393$ and $D=P_1^{\alpha_1}P_2^{\alpha_2}P_3^{\alpha_3}P_4^{\alpha_4}$ $P_5^{\alpha_5}P_6^{\alpha_6}P_7^{\alpha_7}P_8^{\alpha_8}$, where $\alpha_j=b_j\ell^{a_j}$, $\deg P_j=c_j\ell^{b_j}$, $1\leq j\leq 8$. Let
$a_1=0$, $a_2=1$, $a_3=3$, $a_4=3$, $a_5=4$, $a_6=7$, $a_7=8$ and $a_8=9$. Let $d_1=5$, $d_2=7$, $d_3=2$, $d_4=3$, $d_5=2$, $d_6=0$, $d_7=10$ and $d_8=0$. Since ${\rm gcd}(b_jc_j,\ell)=1$, $1\leq j\leq 8$, we can select $c_1=2$, $b_1=b_2=b_4=b_5=b_6=b_7=b_8=c_2=c_4=c_6=c_7=c_8=1$ and $b_3=c_3=c_5=5$. We have $m={\rm max}\{n-a_j-{\rm min}\{n-a_j,d_j\} \smid 1\leq j\leq r\}={\rm max}\{5,2,5,4,4,3,0,1\}=5$ and $t=10-{\rm min}\{10,\nu_\ell(\deg D)\}$ where $\deg D=\ell^5(\ell^3(b_2c_2+b_7c_7\ell^{10}+b_8c_8\ell)+b_1c_1+b_3c_3+\ell(b_4c_4+b_5c_5+b_6c_6\ell))$, and $\nu_\ell(b_1c_1+b_3c_3+\ell(b_4c_4+b_5c_5+b_6c_6\ell))=3$. Thus $\nu_\ell(\deg D)=8$, $\deg D=387459855$ and $t=2$. We have $i_0=3$, $F_3=k\left(\sqrt[\ell^{4}]{P_3^*}\right)$ and
\begin{align*}
F_1&=k\left(\sqrt[\ell^{10}]{P_1P_3^{z_1}}\right),\; F_2=k\left(\sqrt[\ell^{9}]{P_2P_3^{z_2}}\right),\; F_4=k\left(\sqrt[\ell^{7}]{P_4P_{3}^{y_4\ell}}\right),\;\\ F_5&=k\left(\sqrt[\ell^{6}]{P_5P_{3}^{y_5\ell^0}}\right),\; F_6=k\left(\sqrt[\ell^{5}]{P_6^{\ell^{2}}P_{3}^{y_6}}\right),\; F_7=k\left(\sqrt[\ell^{2}]{P_7P_{3}^{y_7\ell^{8}}}\right),\\ F_8&=k\left(\sqrt[\ell^{3}]{P_8^{\ell^{2}}P_{3}^{y_8}}\right),
\end{align*}
with $z_1=-2a\ell^3$, $z_2=-a\ell^5$ and $y_j\equiv -c_j5^{-1}\mod \ell^{10}=-c_j11810\mod 3^{10}$, $4\leq j\leq 8$, where $a5+b\ell^{8}=1$. We may choose $y_4=y_6=y_7=y_8=47239$, $y_5=59048$, $a=-1312$ and $b=1$.
Therefore
\begin{align*}
 \g K=K(\sqrt[3^{10}]{P_1P_3^{70848}}, &\sqrt[3^{9}]{P_2P_3^{318816}},\sqrt[3^{7}]{P_4P_{3}^{141717}},\sqrt[3^{6}]{P_5P_{3}^{59048}},\\&\sqrt[3^{5}]{P_6^{9}P_{3}^{47239}},\sqrt[3^{2}]{P_7P_{3}^{309935079}},\sqrt[3^{3}]{P_8^{9}P_{3}^{47239}}, \sqrt[3^{4-\alpha}]{P_{3}^*}).    
\end{align*}
with
\[
|H'|=3^\alpha=[\F(\sqrt[3^{10}]{-\gamma}):\F(\sqrt[3^8]{-\gamma})].
\]

From \cite[Theorem 4 (1)]{ExtEulerCrit}, $5\notin(\F^*)^\ell$ and $[\F(\sqrt[3^{10}]{-\gamma}):\F(\sqrt[3^8]{-\gamma})]=3^2$. Therefore $\alpha=2$ and $\sqrt[3^{4-\alpha}]{P_{3}^*}=\sqrt[3^{2}]{P_{3}^*}$.
}}
\end{ejm}

\section{On the field extension  $(K_1K_2)_{\eu{ge}}/(K_1)_{\eu{ge}}(K_2)_{\eu{ge}}$}\label{S5}

Let $K_1$ and $K_2$ be function fields over $k=\F(T)$. Let $K:=K_1K_2$ and suppose that $K\subseteq {_n}{\lam N}_m$. Let $E:=E_1E_2$, with $E_i:={_n}(K_i)_m\cap \lam N$, $i=1,2$. We have $E={_n}K_m\cap \lam N$.

\begin{prop}\label{prop5} Let $L_1,L_2\subseteq\lam N$, $L:=L_1L_2$, $L^+:=L\cap\lam N^+$ and $L_i^+:=L_i\cap\lam N^+$, $i=1,2$. Then $[L^+:(L_1)^+(L_2)^+]\smid q-1$.
\end{prop}

\begin{proof}
Let $G:=\Gal(\lam N/E)$, $G_1:=\Gal(\lam N/L_1)$ and $G_2:=\Gal(\lam N/L_2)$. We have $G=G_1\cap G_2$. Let $L^+=L\cap\lam N^+=\lam N^{G\F^*}$, $(L_1)^+=L_1\cap\lam N^+=\lam N^{G_1\F^*}$ and $(L_2)^+=L_2\cap\lam N^+=\lam N^{G_2\F^*}$. Since $G\F^*<G_1\F^*\cap G_2\F^*$, we have 
\[
L_1^+L_2^+= \lam N^{G_1\F^*}\lam N^{G_2\F^*}= \lam N^{G_1\F^*\cap G_2\F^*}\subseteq \lam N^{G\F^*}=L^+.
\]

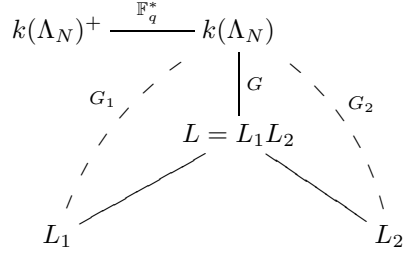
$$\xymatrix{
\lam N^+\ar@{-}[r]^{\F^*}&\lam N\ar@{-}[d]^{G}\ar@/^1pc/@{--}[ddr]^{G_2}&\\
&L=L_1L_2\ar@{-}[dl]\ar@{-}[dr]& \\
L_1\ar@/^1pc/@{--}[uur]^{G_1} && L_2
}$$

Thus
\begin{align*}
    [L^+:L_1^+L_2^+]&=|G_1\F^*\cap G_2\F^*:G\F^*|
    =\frac{|G_1\F^*\cap G_2\F^*|}{|G\F^*|}\\
    &=\frac{\frac{|G_1\F^*||G_2\F^*|}{|G_1G_2\F^*|}}{\frac{|G||\F^*|}{|G\cap\F^*|}}
    =\frac{|G_1\F^*||G_2\F^*||G\cap\F^*|}{(q-1)|G||G_1G_2\F^*|}\\
    &=\frac{|G_1||\F^*||G_2||\F^*||G\cap\F^*|}{(q-1)|G||G_1\cap\F^*||G_2\cap\F^*||G_1G_2\F^*|}\\
    &=(q-1)\frac{|G_1||G_2|}{|G|}\frac{|G\cap \F^*|}{|G_1\cap \F^*||G_2\cap \F^*|}\frac{1}{|G_1G_2\F^*|}\\
    &=(q-1)\frac{|G_1||G_2|}{|G_1\cap G_2|}\frac{1}{|G_1G_2\F^*|}\frac{|G\cap \F^*|}{|G_1\cap \F^*||G_2\cap \F^*|}\\
    &=(q-1)\frac{|G_1 G_2|}{|G_1G_2\F^*|}\frac{|G\cap \F^*|}{|G_1\cap \F^*||G_2\cap \F^*|}\\
    &=(q-1)\frac{1}{|(G_1\cap \F^*)(G_2\cap \F^*)|}\frac{1}{[L_1\cap L_2:L_1^+\cap L_2^+]}.
\end{align*}

Let $\alpha:=|(G_1\cap \F^*)(G_2\cap \F^*)|[L_1\cap L_2:L_1^+\cap L_2^+]$. Since $[L^+:L_1^+L_2^+]\in\Z$, we have $\alpha\smid q-1$, and $[L^+:L_1^+L_2^+]\smid q-1$.
\end{proof}

\begin{prop}\label{prop6}
Let $L\subseteq\lam N$. If $\lam N^+\subseteq L\subseteq\lam N$, then $\g L=L$.
\end{prop}
\begin{proof}
We have $L_{\eu{gex}}=\lam N$ and from \cite[Theorem 2.1]{RockM} it follows that $\g L=L_{\eu{gex}}^+L=\lam N^+L=L$.
\end{proof}

Let $E=E_1E_2\subseteq \lam N$. Let $Y_1$ and $Y_2$ be the groups of Dirichlet characters associated with $(E_1)_{\eu {gex}}$ and $(E_2)_{\eu {gex}}$ respectively. Then, from Leopoldt's Theorem \cite[Proposition 14.4.1]{VilCC}, we have that $Y=Y_1Y_2$ is the group of characters associated with $E_{\eu {gex}}$, where $Y_1Y_2$ is the group of Dirichlet characters associated with the field $(E_1)_{\eu {gex}}(E_2)_{\eu {gex}}$ \cite[Proposition 9.4.33]{VilCC}, i.e, $E_{\eu {gex}}=(E_1)_{\eu {gex}}(E_2)_{\eu {gex}}$. From Proposition \ref{prop5} we obtain that $[E_{\eu {gex}}^+:(E_1)_{\eu {gex}}^+(E_2)_{\eu {gex}}^+]\smid q-1$.

We have the following result.

\begin{prop} Let $K_1,K_2\subseteq {_n}{\lam N}_m$, $K=K_1K_2$ and $E_1$, $E_2$ as before. If $\g K=(K_1)_{\eu {ge}}(K_2)_{\eu {ge}}$, then $\g E=(E_1)_{\eu {ge}}(E_2)_{\eu {ge}}$.
\end{prop}
\begin{proof}
Let $K=K_1K_2\subseteq {_n\lam N_m}$ and $E:=E_1E_2 \subseteq \lam N$. Consider the fields $(\g K)_m$ and $(\g E)_m$. Now, from  \cite[Theorem 2.2]{RockM} we have $(\g E)_m=(\g K)_m$ and $((K_i)_{\eu {ge}})_m=((E_i)_{\eu {ge}})_m$, $i=1,2$. Therefore
\begin{align*}
(\g E)_m=(\g K)_m&=((K_1)_{\eu {ge}}(K_2)_{\eu {ge}})_m=((K_1)_{\eu {ge}})_m((K_2)_{\eu {ge}})_m\\
&=((E_1)_{\eu {ge}})_m((E_2)_{\eu {ge}})_m
=((E_1)_{\eu {ge}}(E_2)_{\eu {ge}})_m.
\end{align*}

Finally, by the Galois correspondence, it follows
\begin{align*}
    \g E=(\g E)_m\cap\lam N=((E_1)_{\eu {ge}}(E_2)_{\eu {ge}})_m\cap\lam N=(E_1)_{\eu {ge}}(E_2)_{\eu {ge}}.
\end{align*}
\end{proof}

The converse of Proposition \ref{prop6} does not hold in general, that is, if $\g E=(E_1)_{\eu{ge}}(E_2)_{\eu{ge}}$, then the equality $\g K=(K_1)_{\eu{ge}}(K_2)_{\eu{ge}}$ may fail.

\begin{ejm} {\rm{Let $K_1=k(\sqrt[\ell]{P_1})$, $K_2=k(\sqrt[\ell]{\gamma P_2})$ and $K=K_1K_2$ be such that $P_1, P_2\in R_T^+$ are different polynomials with $\deg P_1=a$, $1\leq a <\ell$, $\deg P_2=\ell-a$ and $\gamma\notin(\F^*)^\ell$. Then we have $E_i=k(\sqrt[\ell]{(P_i)^*})$, $i=1,2$ and $E=E_1E_2$. If $\chi_{P_i}$ is the Dirichlet character associated with the field $E_i$, $i=1,2$, by Leopoldt's Theorem (\cite[Proposition 14.4.1]{VilCC}), we have $E_i=(E_i)_\eu{ge}$, $i=1,2$ and $E_{\eu {ge}}=E_1E_2=E$. Therefore $\g E=E=E_1E_2=(E_1)_\eu{ge}(E_2)_\eu{ge}$.

Now, since $\p$ is ramified in $K_i/k$, $i=1,2$, we have that $\p$ is of degree $1$ in $K_i$ and $(K_i)_\eu{ge}=K_i$.
On the other hand, from Abhyankar's Lemma, the remification of $\p$ in $K/k$ is equal to ${\rm lcm}[e_\infty(K_1/k),e_\infty(K_2/k)]=\ell$. Since $k(\sqrt[\ell]{\gamma P_1P_2})\subseteq K$, with $\deg P_1P_2=\deg P_1+\deg P_2=\ell$, thus, $\ell\smid \deg P_1P_2$ and $\gamma\notin(\F^*)^\ell$, we obtain $f_\infty(K/k)=\ell$. Also we have $[K:k]=\ell^2=f_\infty(K/k) e_\infty(K/k)$. It follows that $h_\infty(K/k)=1$ and $\deg(S_\infty(K))=f_\infty(K/k)=\ell$. Thus $[K{\ma F}_{q^\ell}:K]=\ell$ and $K\subsetneq K{\ma F}_{q^\ell}\subseteq \g K$, i.e, $\g K\neq K=K_1K_2=(K_1)_\eu{ge}(K_2)_\eu{ge}$.
}}
\end{ejm}

The main result of this section is the following theorem.

\begin{teo} Let $K_1,K_2/k$ be abelian finite field extensions. Then
$$[(K_1K_2)_{\eu{ge}}:(K_1)_{\eu{ge}}(K_2)_{\eu{ge}}]\smid (q-1)^2.$$
\end{teo}
\begin{proof}
We have $K_1,K_2\subseteq {_n\lam N_m}$ for some $m\in\N$, $n\in\N\cup \{0\}$ and $N\in R_T$. Let $K=K_1K_2\subseteq {_n\lam N_m}$. Let $E_i={_n}(K_i)_m\cap \lam N$, $i=1,2$ y $E={_n}K_m\cap\lam N$. Then $E=E_1E_2$. Now, since $E_{\eu{gex}}^+\cap\g E=E_{\eu{gex}}\cap\lam N^+\cap\g E=E_{\eu{gex}}\cap\g E^+=\g E^+$, we have the following Galois square.
\[
\xymatrix{
E_{\eu{gex}}^+\ar@{-}[d]\ar@{-}[r]&E_{\eu{gex}}^+\g E\ar@{-}[d] \\
E_{\eu{ge}}^+\ar@{-}[r]&\g E
}
\]

From \cite[Theorem 2.1]{RockM} we have $\g E=E_{\eu{gex}}^+ E$. Then from the Galois correspondence we obtain $\g E^+=E_{\eu{gex}}^+$. Similarly $(E_i)_{\eu{ge}}^+=(E_i)_{\eu{gex}}^+$, $i=1,2$. We have $[E_{\eu {gex}}^+:(E_1)_{\eu {gex}}^+(E_2)_{\eu {gex}}^+]\smid q-1$. In particular $[E_{\eu {ge}}^+:(E_1)_{\eu {ge}}^+(E_2)_{\eu {ge}}^+]=[E_{\eu {gex}}^+:(E_1)_{\eu {gex}}^+(E_2)_{\eu {gex}}^+]\smid q-1$. Finally, we obtain the following diagram.
\[
\xymatrix{
E_{\eu{ge}}^+\ar@{-}[ddd]\ar@{-}[r]&E_{\eu{ge}}\ar@{-}[dd]\ar@{-}[r]&\g EK\ar@{-}[d]\\
&&(E_1)_{\eu{ge}}(E_2)_{\eu{ge}}K\\
&(E_1)_{\eu{ge}}(E_2)_{\eu{ge}}\ar@{-}[ru]&\\
(E_1)_{\eu{ge}}^+(E_2)_{\eu{ge}}^+\ar@{-}[ru]&&
}
\]

It follows that
\[
[\g EK:(E_1)_{\eu {ge}}(E_1)_{\eu {ge}}K]\smid [E_{\eu {ge}}:(E_1)_{\eu {ge}}(E_2)_{\eu {ge}}]\smid [E_{\eu {gex}}^+:(E_1)_{\eu {ge}}^+(E_2)_{\eu {ge}}^+].
\]

Therefore $[\g EK:(E_1)_{\eu {ge}}(E_1)_{\eu {ge}}K]\smid q-1$.

On the other hand, from \cite[Theorem 2.2]{RockM} we have that the following extensions $(E_i)_{\eu {ge}}K_i/(K_i)_{\eu {ge}}$, $i=1,2$, are constant extension fields of order $|H_i|\smid q-1$, $i=1,2$, were $H_i$ is the decomposition group of $S_\infty(K_i)$ in $(E_i)_{\eu {ge}}K_i$, $i=1,2$. Also, if $t_i:=\deg S_\infty(K_i)$, then $(E_i)_{\eu {ge}}K_i=(K_i)_{\eu {ge}}{\ma F}_{q^{t_i|H_i|}}$, $i=1,2$. It follows that 
\begin{align*}
    (E_1)_{\eu {ge}}K_1(E_2)_{\eu {ge}}K_2&=(K_1)_{\eu {ge}}{\ma F}_{q^{t_1|H_1|}}(K_2)_{\eu {ge}}{\ma F}_{q^{t_2|H_2|}}\\
    &=(K_1)_{\eu {ge}}(K_2)_{\eu {ge}}{\ma F}_{q^{{\rm lcm}[t_1|H_1|,t_2|H_2|]}}\\
    &\subseteq (K_1)_{\eu {ge}}(K_2)_{\eu {ge}}{\ma F}_{q^{{\rm lcm}[t_1(q-1),t_2(q-1)]}}\\
    &= (K_1)_{\eu {ge}}(K_2)_{\eu {ge}}{\ma F}_{q^{{\rm lcm}[t_1,t_2](q-1)}}. 
\end{align*}

We have ${\ma F}_{q^{{\rm lcm}[t_1,t_2]}}\subseteq (K_1)_{\eu {ge}}(K_2)_{\eu {ge}}$, i.e, the field $(K_1)_{\eu {ge}}(K_2)_{\eu {ge}}{\ma F}_{q^{{\rm lcm}[t_1,t_2](q-1)}}$ is an extension of constants of $(K_1)_{\eu {ge}}(K_2)_{\eu {ge}}$ of degree at most $q-1$. We have
\[
(K_1)_{\eu {ge}}(K_2)_{\eu {ge}}\subseteq (E_1)_{\eu {ge}}(E_2)_{\eu {ge}}K\subseteq (K_1)_{\eu {ge}}(K_2)_{\eu {ge}}{\ma F}_{q^{{\rm lcm}[t_1,t_2](q-1)}}.
\]

Therefore, $(E_1)_{\eu {ge}}(E_2)_{\eu {ge}}K/(K_1)_{\eu {ge}}(K_2)_{\eu {ge}}$ is an extension  of constants of degree at most $(q-1)$. Finally, we have $[\g EK:(E_1)_{\eu {ge}}(E_2)_{\eu {ge}}K]\smid [\g E:(E_1)_{\eu {ge}}(E_2)_{\eu {ge}}]\smid q-1$. Thus
\begin{align*}
[\g K:(K_1)_\eu{ge}(K_2)_\eu{ge}]&\smid[\g EK:(E_1)_{\eu {ge}}(E_2)_{\eu {ge}}K][(E_1)_{\eu {ge}}(E_2)_{\eu {ge}}K:(K_1)_{\eu {ge}}(K_2)_{\eu {ge}}]\\
&\smid (q-1)^2.
\end{align*}
\end{proof}

Let $H$, $H_1$ and $H_2$ be the decomposition groups of $\p$ in 
$\g EK/K$, $(E_1)_{\eu {ge}}K_1/K_1$ and $(E_2)_{\eu {ge}}K_2/K_2$ respectively and let $H':=H|_{\g E}$, $H'_i:=H_i|_{(E_i)_{\eu {ge}}}$, $i=1,2$.

\begin{prop} With the previous notation, we have
\[
(E_1)_{\eu {ge}}^{H'_1}(E_2)_{\eu {ge}}^{H'_2}\subseteq E_{\eu {ge}}^{H'}.
\]
\end{prop}
\begin{proof}
Since $K=K_1K_2$, then $(K_i)_\eu{ge}\subseteq \g K$, $i=1,2$. Because $\g K=E_\eu{ge}^{H'}K$ and $(K_i)_\eu{ge}=(E_i)_\eu{ge}^{H'_i}K_i$, $i=1,2$, it follows that $(E_i)_\eu{ge}^{H'_i}K\subseteq E_\eu{ge}^{H'}K$, $i=1,2$. Now, since $E_i\subseteq (E_i)_{\eu{ge}}$, from the Galois correspondence it follows that $E_i^{H'_i}\subseteq (E_i)_{\eu{ge}}^{H'_i}$, $i=1,2$. Hence $E_i^{H'_i}K\subseteq (E_i)_\eu{ge}^{H'_i}K\subseteq E_\eu{ge}^{H'}K$. Thus, we have the following Galois square.
\[
\xymatrix{
\g E\ar@{-}[d]_{H'}\ar@{-}[r]&\g EK\ar@{-}[d]^H\\
\g E^{H'}\ar@{-}[d]\ar@{--}[r]&\g E^{H'}K=\g K\ar@{-}[d]\\
\g E\cap (E_i)_{\eu {ge}}^{H'_i}K\ar@{-}[d]\ar@{--}[r]&(E_i)_{\eu {ge}}^{H'_i}K\ar@{-}[d]\\
\g E\cap (E_i)^{H'_i}K_i\ar@{-}[d]\ar@{--}[r]&(E_i)^{H'_i}K\ar@{-}[d]\\
\g E\cap K\ar@{-}[r]&K
}
\]

From the Galois correspondence, we have $\g E\cap (E_i)^{H'_i}K\subseteq \g E\cap (E_i)_{\eu {ge}}^{H'_i}K\subseteq \g E\cap \g K=\g E^{H'}$, where $(E_i)^{H'_i}\subseteq \g E\cap (E_i)^{H'_i}K$ and $(E_i)_{\eu {ge}}^{H'_i}\subseteq \g E\cap (E_i)_{\eu {ge}}^{H'_i}K$. Therefore $\;(E_i)_{\eu {ge}}^{H'_i}\subseteq\g E^{H'}$, $i=1,2$.
\end{proof}

\begin{teo} With the previous notation, we have $|H|\smid {\rm lcm}[|H'_1|,|H'_2|]$.
\end{teo}
\begin{proof}
We have that $\p$ is totally ramified in the extension $\g E/E_\eu{ge}^{H'}$ and that $\p$ is totally decomposed in $\g E/(E_1)_\eu{ge}(E_2)_\eu{ge}$ since 
\[
E=E_1E_2\subseteq (E_1)_\eu{ge}(E_2)_\eu{ge}\subseteq \g E.
\]

It follows that $\g E=E_\eu{ge}^{H'}(E_1)_\eu{ge}(E_2)_\eu{ge}$, since if $L:=E_\eu{ge}^{H'}(E_1)_\eu{ge}(E_2)_\eu{ge}$, in the extension $\g E/L$, $\p$ is totally ramified and totally decomposed.

Now, $(E_1)_{\eu{ge}}/(E_1)_\eu{ge}^{H'_1}$ is totally ramified in $\p$ with ramification index $|H'_1|$, so that
$$\e{(E_1)_\eu{ge}(E_2)_\eu{ge}/(E_1)_\eu{ge}^{H'_1}(E_2)_\eu{ge}}\smid |H'_1|$$
\[
\xymatrix{
(E_1)_\eu{ge}\ar@{-}[r]\ar@{-}[d]_{|H'_1|}&(E_1)_\eu{ge}(E_2)_\eu{ge}^{H'_2}\ar@{-}[d]^{e_{\infty,1}\smid |H'_1|}\\
(E_1)_\eu{ge}^{H'_1}\ar@{-}[r]&(E_1)_\eu{ge}^{H'_1}(E_2)_\eu{ge}^{H'_2}
}
\]

Similarly $\e{(E_1)_\eu{ge}(E_2)_\eu{ge}/(E_1)_\eu{ge}(E_2)_\eu{ge}^{H'_2}}\smid |H'_2|\smid \e{(E_2)_{\eu{ge}}/(E_2)_\eu{ge}^{H'_2}}=[(E_2)_{\eu ge}:(E_2)_\eu{ge}^{H'_2}]$ 
\[
\xymatrix{
(E_2)_\eu{ge}\ar@{-}[r]\ar@{-}[d]_{|H'_2|}&(E_1)_\eu{ge}^{H'_1}(E_2)_\eu{ge}\ar@{-}[d]^{e_{\infty,2}\smid |H'_2|}\\
(E_2)_\eu{ge}^{H'_2}\ar@{-}[r]&(E_1)_\eu{ge}^{H'_1}(E_2)_\eu{ge}^{H'_2}
}
\]

We have $(E_1)_\eu{ge}(E_2)_\eu{ge}=((E_1)_\eu{ge}(E_2)_\eu{ge}^{H'_2})((E_1)_\eu{ge}^{H'_1}(E_2)_\eu{ge})$ and the following diagram
\[
\xymatrix{
&(E_1)_\eu{ge}(E_2)_\eu{ge}\ar@{-}[dr]&\\
(E_1)_\eu{ge}(E_2)_\eu{ge}^{H'_2}\ar@{-}[ur]&&(E_1)_\eu{ge}^{H'_1}(E_2)_\eu{ge}\ar@{-}[dl]^{e_{\infty,2}\smid |H'_2|}\\
&(E_1)_\eu{ge}^{H'_1}(E_2)_\eu{ge}^{H'_2}\ar@{-}[ul]^{e_{\infty,1}\smid |H'_1|}&\\
}
\]

From Abhyankar's Lemma
\[
\e{(E_1)_\eu{ge}(E_2)_\eu{ge}/(E_1)_\eu{ge}^{H'_1}(E_2)_\eu{ge}^{H'_2}}=e_0:={\rm lcm}[e_{\infty,1},e_{\infty,2}]\smid {\rm lcm}[|H'_1|,|H'_2|].
\]

Finally, we have $\p$ is totally ramified in $\g E/\g E^{H'}$ and totally decomposed in $\g E/(E_1)_\eu{ge}(E_2)_\eu{ge}$. 
\[
\xymatrix{
E_\eu{ge}^{H'}\ar@{-}[r]^{|H'|}\ar@{-}[d]&E_\eu{ge}\ar@{-}[d]^{\substack{\p\\
\text{totally decomposed}}}\\
(E_1)_\eu{ge}^{H'_1}(E_2)_\eu{ge}^{H'_2}\ar@{-}[r]_{e_0}&(E_1)_\eu{ge}(E_2)_\eu{ge}
}
\]

Thus $|H'|\smid [\g E:(E_1)_\eu{ge}(E_2)_\eu{ge}]\, e_0$. Thus $|H'|\smid e_0$ and $e_0\smid {\rm lcm}[|H'_1|,|H'_2|]$. Therefore $$|H'|\smid {\rm lcm}[|H'_1|,|H'_2|].$$
\end{proof}

\begin{teo} Let $K_1,K_2\subseteq {_n}{\lam N}_m$, $K=K_1K_2$, $F:={_n}K\cap {\lam N}_m$ and $F_i:={_n}K_i\cap {\lam N}_m$, $i=1,2$. Then 
\[
\g K=(K_1)_\eu{ge}(K_2)_\eu{ge}\textit{ if and only if } \g F=(F_1)_\eu{ge}(F_2)_\eu{ge}.
\]
\end{teo}
\begin{proof}
Since $K=K_1K_2$ we have ${_n}K={_n}{(K_1K_2)}={_n}{(K_1)}{_n}{(K_2)}$. From the Galois correspondence it follows that $F=F_1F_2$.

From \cite[Theorem 4.4]{MaRzVi2013} we have that $\g K=M\g F$ where $M=K\lam N_m\cap L_n$, $\g F$ is the genus field of $F$. We have $M\cap \g F\subseteq L_n\cap{\lam N}_m =k$. Then $\g K=M\g F$ and $(K_i)_\eu{ge}=M_i(F_i)_\eu{ge}$, where $M_i=K_i{\lam N}_m\cap L_n$, $i=1,2$. If $V$ is the first ramification group of $\p$ in $\g K/k$, then $\g F=K_{\eu{ge}}^V$, and we have that $[\g K:\g F]=[M:k]=|V|$.
\[
\xymatrix{
F_{{\eu {ge}}}\ar@{-}[rr]^V\ar@{-}[d]&&\g K=M\g F\ar@{-}[d]\\
F\ar@{-}[rr]\ar@{-}[dd]_{\substack{\p \text{\ is tamely}\\
\text{ramified}}}&&FM\ar@{-}[dd]^{
\substack{\S M\text{\ is tamely}\\ \text{ramified}}}\\
&K\ar@{-}[ru]\ar@{-}[dl]\\ k\ar@{-}[rr]^V&&M
}
\]

Then, by the Galois correspondence we have $\g F=\g K\cap \lam N_m$ and $[\g K:M]=[\g F:k]$.

Now, first suppose that $\g K=(K_1)_\eu{ge}(K_2)_\eu{ge}$. Since
\[
M\g F=\g K=(K_1)_\eu{ge}(K_2)_\eu{ge}=M_1(F_1)_\eu{ge}M_2(F_2)_\eu{ge}=M_1M_2((F_1)_\eu{ge}(F_2)_\eu{ge}).
\]

Note that $M{\lam N}_m=M_1{\lam N}_m M_2{\lam N}_m=M_1M_2{\lam N}_m$. Therefore, by the Galois correspondence $M=M_1M_2$. Thus $M\g F=M_1M_2((F_1)_\eu{ge}(F_2)_\eu{ge})=M((F_1)_\eu{ge}(F_2)_\eu{ge})$. It follows that $\g F=(F_1)_\eu{ge}(F_2)_\eu{ge}$.

Conversely, assume that $\g F=(F_1)_\eu{ge}(F_2)_\eu{ge}$. Then 
\[
\g K=M\g F=M_1(F_1)_\eu{ge}M_2(F_2)_\eu{ge}=(K_1)_\eu{ge}(K_2)_\eu{ge}.
\]
\end{proof}

\end{document}